\documentclass[10pt]{IEEEtran4PSCC}
\usepackage[all]{xy}
\usepackage{color}
\usepackage{algorithm}
\usepackage{algpseudocode}
\usepackage{subfigure}
\usepackage{bm}
\usepackage{graphicx}
\usepackage{color}
\usepackage{wrapfig}
\usepackage{epsf}
\usepackage{times,mathptm}
\usepackage{url}
\usepackage{amssymb, amsmath, mathtools}
\usepackage{nomencl}
\usepackage{comment}

\makenomenclature
\newtheorem{theorem}{Theorem}

\newtheorem{corollary}{Corollary}
\newtheorem{lemma}{Lemma}
\newtheorem{proof}{Proof}
\definecolor{RED}{rgb}{0.6,0.,0.}
\definecolor{BLUE}{rgb}{0.,0.,0.6}
\definecolor{GREEN}{rgb}{0.,0.6,0.}
\definecolor{MALINA}{rgb}{0.6,0.,0.6}
\definecolor{YELLOW}{rgb}{0.8,0.8,0}

\begin{document}

\title{Estimating Distribution Grid Topologies:\\ A Graphical Learning based Approach}
\author{
\IEEEauthorblockN{Deepjyoti~Deka\dag, Scott~Backhaus\dag, and Michael~Chertkov\dag}
\IEEEauthorblockA{(\dag) Los Alamos National Laboratory\\
 Los Alamos, New Mexico, United States\\
\{deepjyoti, backhaus, chertkov\}@lanl.gov}
}

\maketitle
\begin{abstract}
Distribution grids represent the final tier in electric networks consisting of medium and low voltage lines that connect the distribution substations to the end-users/loads. Traditionally, distribution networks have been operated in a radial topology that may be changed from time to time. Due to absence of a significant number of real-time line monitoring devices in the distribution grid, estimation of the topology/structure is a problem critical for its observability and control. This paper develops a novel graphical learning based approach to estimate the radial operational grid structure using voltage measurements collected from the grid loads. The learning algorithm is based on conditional independence tests for continuous variables over chordal graphs and has wide applicability. It is proven that the scheme can be used for several power flow laws (DC or AC approximations) and more importantly is independent of the specific probability distribution controlling individual bus's power usage. The complexity of the algorithm is discussed and its performance is demonstrated by simulations on distribution test cases.
\end{abstract}

\section{Introduction}
\label{sec:intro}
The power grid comprises of the set of transmission lines that transfer power from generators to the end users. Due to operational ease, the grid is separated into different tiers or levels: transmission grid consisting of high voltage lines connecting the generators to the distribution substations, and distribution grid consisting of the medium and low voltage lines that connect the distribution substations to the final loads. Structurally, the distribution grid is distinguished by its radial (tree-like) operational topology with the substation at the root and loads positioned along the non-root nodes of the tree. In reality \cite{distgridpart1}, the actual set of lines comprising the distribution grid form a loopy graph but during operation, several of these lines are disconnected by open switches to create a tree structure (see Fig.~\ref{fig:city}).

With the advent of the smart grid, distributed control and optimization of distribution grid operations have become a reality through deployment of smart loads, deferrable loads, energy storage as well as distributed generation (like roof-top solar, wind). Optimal operations in the distribution grid depend on the correct estimation of its bus/node states (voltage and power consumption) and its operational radial topology. However, lines in the grid still suffer from limited real time metering that hinders the grid operator from learning the true topology \cite{hoffman2006practical}. In recent times, smart meters like PMUs \cite{phadke1993synchronized} have been built at the grid buses or individual households to relay high fidelity measurements of bus/nodal state in real time for demand response. The goal of this paper is to judiciously use the nodal real-time measurements to estimate the operational topology in the grid. Note that due to the exponential number of possible radial trees that can be created from a given loopy graph, a brute force structure learning scheme that tests the correctness of nodal measurements is computationally prohibitive. Efficient learning of the correct operational topology thus needs to utilize relations induced by the radial structure on the available measurements that can be checked with relative ease. In this work, we determine a specific graphical model based characterization of the probability distribution of bus voltages in a radial network and use it to develop a novel framework for structure estimation. Our learning framework is very general and hence able to operate under different operational conditions and power flow models that have been discussed in prior literature.

\subsection{Prior Work}
In the past, several efforts have been made to analyze the problem of topology estimation in power grids. \cite{kekatos2013grid} uses a maximum likelihood estimator with low-rank and sparsity regularizers to estimate the grid structure from electricity prices. The authors of \cite{he2011dependency} use a Markov random field model for bus phase angles and use it to build a dependency graph to detect faults in grids. \cite{bolognani2013identification} presents a topology identification algorithm for distribution grids with linear power flow model and constant $R/X$ (resistance to reactance) ratio for transmission lines. The algorithm uses the signs of elements in the inverse covariance of voltage measurements to build the operational tree. In \cite{sharon2012topology}, a machine learning (ML) topology estimate with approximate schemes is used in a distribution grid with Gaussian loads. Our previous work, \cite{distgridpart1} presents greedy structure and parameter learning algorithms using provable trends in second order moments of voltage magnitudes determined by a linear coupled approximation for lossless AC power flow \cite{89BWa,89BWb}. Subsequent work \cite{distgridpart2} extends this approach to reconstruct the operational tree with incomplete information where significant data is missing. Sets of line flow measurements have been used for topology estimation using maximum likelihood tests in \cite{ramstanford}. Other approaches \cite{berkeley} include machine learning schemes that compare available time-series observations from smart meters with database of permissible signatures to identify topology changes. Similar envelope comparison schemes have been used in parameter estimation as well \cite{sandia1, sandia2}. In the broader category of general random variables, graphical models \cite{wainwright2008graphical} provide a graph-based technique to understand the relationship/interaction between different variables and have been used to study languages, genetic networks, social interactions, decoding in communication schemes etc. To learn the graphical model structure, estimation schemes used include maximum likelihood schemes \cite{wainwright2008graphical} along with restrictions like sparsity \cite{ravikumar2010high}. Other schemes learn the graph by determining the neighbors of every node by conditional mutual information based tests \cite{anandkumar2011high, netrapalli2010greedy}. It is worth mentioning that a majority of work on graphical model learning is designed for discrete values variables or Gaussian random variables, where computation of conditional mutual information is relative easy. As power grid voltage and load profiles may pertain to arbitrary distributions over continuous random variables, detection tests will need to be general in nature. In work specific to continuous random variables, \cite{gretton2007kernel,zhang2012kernel} have designed kernel-based conditional independence tests for determining causality in directed graphs. We use such techniques within our learning scheme in this paper.

\subsection{Contribution}
In this paper, we develop a graphical model learning framework to determine the operational radial structure of the grid using nodal voltage measurements (voltage magnitudes and/or phasor measurements). In particular, we show that under standard assumptions of power consumption and common power flow models (DC and AC approximation),\emph{ the probability distribution of continuous valued nodal voltages can be described by a specific chordal graphical model}. This specific factorization motivates our structure learning algorithm where conditional independence tests on node quartets ($4$ variables each) are used to learn the operational edges. The graphical model based approach helps bridge previous work in the power domain designed for specific flow models with separate research in machine learning for general graphs. We present conditions under which the graphical model based scheme is applicable for several power flow laws. Our learning algorithm is independent of the exact probability distribution for each node's power usage and voltage profile, and also does not require knowledge on line impedances for estimating the structure. We show the polynomial computational complexity of our learning algorithm and present its performance by simulations on test radial networks.

The next section presents a brief discussion of distribution grid topology and a description of common power flow models. Following it, Section \ref{sec:graphicalmodel} analyzes the graphical model of power grid voltage measurements and its specific structure. We present conditional independence based properties of the distribution of voltage measurements in Section \ref{sec:learning} and use it to develop our learning algorithm. Section \ref{sec:conditional} discusses the computation of conditional independence tests for continuous random variables used in our learning algorithm. Simulation results of our learning algorithm on a test radial network are presented in Section \ref{sec:simulations}. Finally conclusions and future work are included in Section \ref{sec:conclusion}.

\section{Distribution Grid: Structure and Power Flows}
\label{sec:structure}
\textbf{Radial Structure}: We represent the distribution grid by the graph ${\cal G}=({\cal V},{\cal E})$, where ${\cal V}$ is the set of buses/nodes of the graph and ${\cal E}$ is the set of undirected lines/edges. The operational grid is determined by closed switches (operational lines) and has a `radial' structure as shown in Fig.~\ref{fig:city}. The operational grid is a collection of $K$ disjoint trees, $\cup_{i=1,\cdots,K}{\cal T}_i$. Each tree ${\cal T}_i$ spans a subset of the nodes ${\cal V}_{\cal T}$ connected by the set of operational edges ${\cal E}_{\cal T}$. Note that each tree has a substation at the root node. We denote nodes by alphabets $a$, $b$ and so on. The undirected edge connecting nodes $a$ and $b$ is denoted by $(ab)$.
\begin{figure}[!bt]
\centering
\includegraphics[width=0.40\textwidth, height =.30\textwidth]{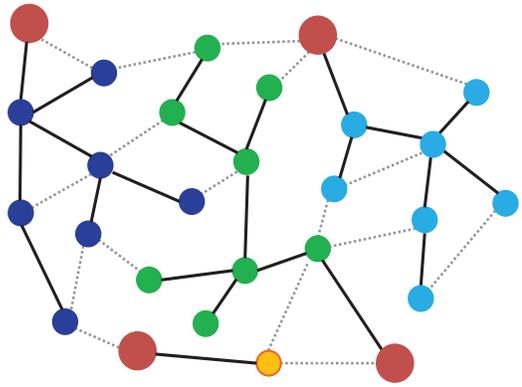}
\caption{Example of a distribution grid fed from $4$ substations, with substations represented by large red nodes. The operational grid is formed by solid lines (black). Dotted grey lines represent open switches. Load nodes within each tree are marked with the same color.
\label{fig:city}}
\end{figure}

In the rest of the paper, we limit our analysis to grids where the operational structure consists of one tree $\cal T$. The analysis can be directly extended to a general case where multiple disjoint trees are present as noted in comments in later sections. Further, we assume that the substation (root node) is connected by a known edge to one load node. The operational edges between the load nodes are unknown to the observer who has access only to load voltage measurements, but no flow or power injection measurements in the grid.

\textbf{Power Flow (PF) Models}: Let $z_{ab}=r_{ab}+i x_{ab}$ denote the complex impedances of line $(ab)$ in the distribution grid ($i^2=-1$) with resistance $r_{ab}$ and reactance $x_{ab}$. By Kirchhoff's laws, the complex valued PF equations governing the flow of power out of a node $a$ in a tree ${\cal T}$ is given by
\begin{align}
& P_a =p_a+i q_a=\underset{b:(ab)\in{\cal E}_{\cal T}}{\sum} V_a(V_a^* - V_b^*)/z_{ab}^*\label{P-complex}\\
&= \underset{b:(ab)\in{\cal E}_{\cal T}}{\sum}\frac{v_a^2-v_a v_b\exp(i\theta_a-i\theta_b)}{z_{ab}^*}\label{P-complex1}
\end{align}
where the real valued scalars, $v_a$, $\theta_a$, $p_a$ and $q_a$ denote the voltage magnitude, voltage phase, active and reactive power injection respectively at node $a$. $V_a (= v_a\exp(i\theta_a))$ and $P_a$ denote the nodal complex voltage and injection respectively. During normal operation, it is often assumed that the grid is lossless and the net power (generation minus demand) in the grid is zero. In that case, one bus in each tree can be considered as reference bus, with its power injection given by negative of the sum of injections at all other buses. Further the voltage and phase at the reference bus is considered at a fixed reference value without a loss of generality. In subsequent sections, we denote the substation root node in the operational tree as the reference node. Abusing notation, we use ${\cal V}_{\cal T}$ to denote the non-reference nodes/buses in the grid. Under the lossless assumption, there are several well-known relaxations to the PF equations as described below.

\textbf{DC model} \cite{abur2004power}: In this popular approximation, lines are considered to be purely inductive and voltage magnitudes are assumed to be constant at unity. Further, phase differences between neighboring lines are assumed to be small. This results in the following linear relation between active power flows and phase angles.
\begin{align}
\forall a\in {\cal V}_{\cal T}: &p_a = \sum_{b:(ab)\in{\cal E}_{\cal T}} \beta_{ab}(\theta_a-\theta_b) \nonumber\\
\text{In vector form,~} &p = H_{\beta}\theta \label{PF_DC}
\end{align}
Here, $H_{\beta}$ is the reduced weighted Laplacian matrix for tree ${\cal T}$ with edge weights given by line susceptances $\beta$ ($\beta_{ab}=\frac{x_{ab}}{x_{ab}^2+r_{ab}^2}$ for edge $(ab)$). The reduction is achieved by removing the row and column corresponding to the reference bus from the weighted Laplacian matrix.

\textbf{Linear Coupled (LC) model} \cite{distgridpart1,distgridpart2}: In this model, the PF Eqs.~(\ref{P-complex}) are linearized jointly over voltage phase difference and magnitude deviations ($v_a -1=\varepsilon_a$) from the reference voltage, both of which are considered to be small. We arrive at the following set of Linear-Coupled (LC) approximations, that are a generalization of the DC model.:
\begin{align}
p_a&=\underset{b:(ab)\in{\cal E}_{\cal T}}{\sum}\left(\beta_{ab}(\theta_a-\theta_b)+
g_{ab}(\varepsilon_a-\varepsilon_b)\right),
\label{PF_LPV_p}\\
q_a&=\underset{b:(ab)\in{\cal E}_{\cal T}}{\sum}
\left(-g_{ab}(\theta_a-\theta_b)+
\beta_{ab}(\varepsilon_a-\varepsilon_b)\right)\label{PF_LPV_q}\\
\text{where~} & g_{ab}\doteq\frac{r_{ab}}{x_{ab}^2+r_{ab}^2}, \beta_{ab}\doteq\frac{x_{ab}}{x_{ab}^2+r_{ab}^2} \label{g}
\end{align}

\textbf{LinDistFlow Model} \cite{89BWa,89BWb,89BWc}: This model is ideally made for analysis in distribution grids. Here nodal power injections are first expressed in terms of directed line flows, which are subsequently related to nodal voltages.
\begin{align}
&p_{a\to b}\approx p_b+\sum_{\substack{(bc)\in{\cal E}_{{\cal T}}\\c \neq a}} p_{b\to c},
q_{a\to b}\approx q_b+\sum_{\substack{(bc)\in{\cal E}_{{\cal T}}\\c \neq a}} q_{b\to c},
\label{lin_DF_q}\\
&\varphi_b = \varphi_a-2\left(r_{ab}p_{a\to b}+x_{ab} q_{a\to b}\right), \quad \varphi_a\equiv v_a^2
\label{lin_DF_v}
\end{align}
Note that due to the radial structure, there exist an invertible map between nodal injections and line flows and hence voltages can be efficiently computed. Further, as shown in \cite{distgridpart1}, if small deviations in voltage magnitude are assumed for the LinDistFlow model, we get the LC Eqs.~(\ref{PF_LPV_p}) and (\ref{PF_LPV_q}).

Other than being lossless, another characteristic of all the above mentioned models (DC, LC and LinDistFlow) is that they represent nodal power injections as an invertible function of nodal voltages (relative to the reference bus). This property is key as the basic PF Eqs.~(\ref{P-complex}) can have multiple feasible voltage profiles resulting in the same injection profile. In the next section, we show the use of PF and its lossless approximations in determining the graphical model of nodal voltages.

\section{Graphical Model from Power Flows}
\label{sec:graphicalmodel}
Our aim here is to obtain a graphical model representation for the probability distribution of voltages at all non-substation nodes in the distribution grid tree ${\cal T}$ (see Fig.~\ref{fig:junction1}). We make the following assumptions regarding power flows in the distribution tree:\\
\textbf{Assumption 1}: The random variables representing the power consumption at two load nodes (non-substation) of tree ${\cal T}$ are independent of each other and attain a steady distribution before the radial configuration changes.\\
\textbf{Assumption 2}: There is an invertible relation between power injections and nodal voltages within the domain of consideration.

Note that the first is a realistic assumption over short to medium intervals (minutes to hours) as fluctuations determining the probability distribution of load at different nodes will be independent and uncorrelated. The second assumption is commonly assumed in stability and convergence studies in power system applications \cite{farivar2013branch}. Further they are true for the power flow approximations described in the previous section. Under these assumptions, the continuous random vector $P$ representing the power consumption at the non-substation nodes ${\cal V}_{\cal T}$ in tree ${\cal T}$ has the following probability distribution function (p.d.f.):

\begin{align}
{\cal P}(P)=\prod_{a\in {\cal V}_{\cal T}}{\cal P}_a(P_a) = \prod_{a\in{\cal V}_{\cal T}}{\cal P}_a\left(\smashoperator[r]{\sum_{b:\{a,b\}\in{\cal E}_{\cal T}}}V_a(V_a^* - V_b^*)/z_{ab}^*\right) \label{GM_P}
\end{align}
where Eq.~(\ref{GM_P}) follows from Eq.~(\ref{P-complex}). ${\cal P}_a$ is the p.d.f. for injection at node $a$.
The distribution of complex voltage vector $V$ for the non substation nodes, using standard rules for p.d.f. of inverse functions, is given by
\begin{align}
{\cal P}(V) = \frac{1}{|J_P(V)|}\prod_{a\in{\cal V}_{\cal T}}{\cal P}_a\left(\smashoperator[r]{\sum_{b:\{a,b\}\in{\cal E}_{\cal T}}}V_a(V_a^* - V_b^*)/z_{ab}^*
\right) \label{GM_V}
\end{align}
Here $|J_P(V)|$ is the determinant of the Jacobian matrix $J_P(V) = \frac{\partial V}{\partial P}$ for the invertible transformation from complex vectors $P$ to $V$. We now create the graphical model using Eq.~(\ref{GM_V}).

\textbf{Graphical Model}: For a $n$ dimensional random vector $X = [X_1, X_2,..X_n]^T$, we create the undirected graphical model ${\cal GM}$ \cite{wainwright2008graphical} with vertex set ${\cal V_{GM}}$ such that each node represents one variable. For every node $i$, its graph neighbors form the smallest set of nodes $N(i) \subset {\cal V_{GM}}- {i}$ such that for any node set $C \subset {\cal V_{GM}} - {i}-N(i)$, ${\cal P}(X_i|X_ {N(i)},X_C) = {\cal P}(X_i|X_ {N(i)})$. Here $X_C$ represents the random variables corresponding to nodes in set $C$. Thus every node is conditionally independent of all other nodes given its graph neighborhood. Similarly, if deletion of a set of nodes $C$ separates the graphical model ${\cal GM}$ into two disjoint sets $A$ and $B$, then each node in set $A$ is conditionally independent of a node in $B$ given all nodes in $C$.

Observe the probability distribution of nodal voltages in Eq.~(\ref{GM_V}). Ignoring the Jacobian determinant, node $a$ appears in factors ${\cal P}_a$ and all terms ${\cal P}_b$, where $b$ is a tree neighbor of $a$. This is because each product term include voltages corresponding to a node and all its neighbors. Let $N_2(a)$ be the set of all nodes in tree ${\cal T}$ that are at a distance $1$-hop (neighbors) and $2$-hops (neighbors of neighbors) from $a$. If the distribution ${\cal P}(V)$ is conditioned on voltages corresponding to set $N_2(a)$, ignoring the Jacobian determinant, the rest can be product separated into terms containing only $a$ as a variable and terms excluding $a$. If the Jacobian determinant is product separable as well into nodes inside and outside of some set $B \cup \{a\}$ where $B \subset N_2(a)$, voltage at $a$ will be conditionally independent of the rest of the nodes given voltages in its 2-hop neighborhood. Thus, the following hold:
\begin{lemma} \label{lemma1}
For each node $a$ in grid tree ${\cal T}$, let the determinant of the Jacobian $J_P$ for the transformation between power injections and voltages be product separable into terms involving nodes in $B \cup {a}$ where $B \subset N_2(a)$, the two-hop neighborhood of $a$ in ${\cal T}$. Then the Graphical Model $\cal{GM}$ for the p.d.f. of load voltages in Eq.~(\ref{GM_V}) is given by inserting additional edges between every pair of neighbors of each node in tree ${\cal T}$.
\end{lemma}
The proof follows immediately from the discussion on conditional independence and definition of graphical model. Fig.~\ref{fig:junction2} shows an example construction of a graphical model. Each node in $\cal GM$ represents the random variable for voltage at its corresponding node in $\cal T$. Note that the graphical model is loopy as neighbors of each node in the original grid tree ${\cal T}$ are connected by edges in $\cal GM$. These additional edges create clique sets (complete subgraphs), each composed of a non-leaf node and its immediate neighbors from tree ${\cal T}$, as depicted in Fig.~\ref{fig:junction2}. In the terminology of graph theory, $\cal GM$ is chordal (every cycle of size greater than $3$ has a chord) and its p.d.f. has a Junction Tree Factorization \cite{wainwright2008graphical}. We leave a discussion on this factorization for future work.

A special case of a separable Jacobian determinant is one with constant determinant. We know that Jacobian matrices of linear functions have constant determinants. We thus have the following corollary to Lemma~\ref{lemma1}.
\begin{corollary} \label{corollary1}
If the nodal power injections in a distribution tree ${\cal T}$ can be represented as a linear function of the nodal voltages, then the graphical model $\cal GM$ for the nodal voltages is generated by adding edges between every pair of neighbors of each node in tree ${\cal T}$.
\end{corollary}
In particular, as DC, LC models and LinDistFlow approximation (reactive power and square of voltage magnitudes if active power is kept constant) mentioned in the previous section each represent a linear relation, they generate the graphical model structure given by Corollary~\ref{corollary1} and Fig.~\ref{fig:junction2}. It needs to be mentioned that outside of Corollary~\ref{corollary1}, other invertible voltage-power relations may exist that satisfy the separability of Jacobian determinants mentioned in Lemma~\ref{lemma1}. Further, the graphical model characterization does not depend on the exact distribution characterizing each node's power injections. However convergence of empirical results based on the graphical model will depend on the distributions considered.

\begin{figure}[!bt]
\centering
\subfigure[]{\includegraphics[width=0.22\textwidth,height =.18\textwidth]{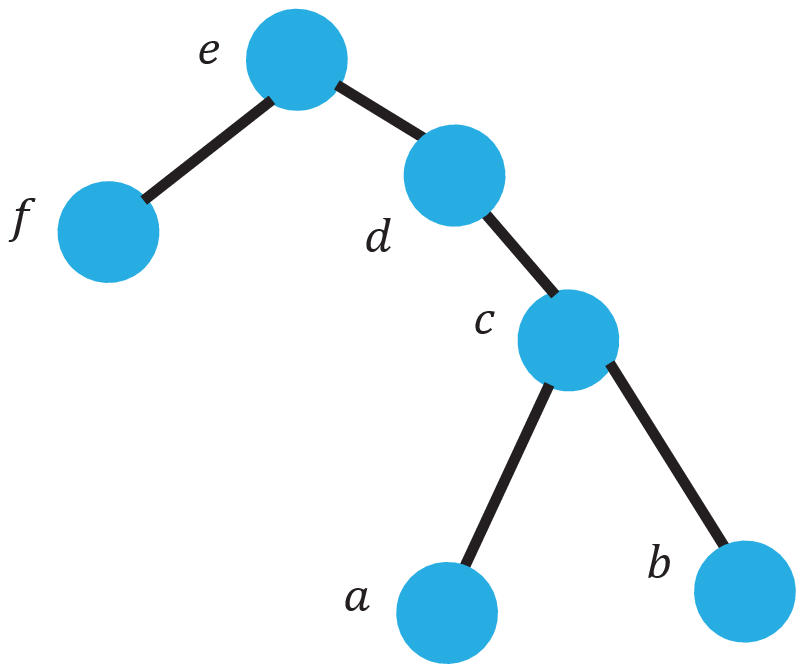}\label{fig:junction1}}
\subfigure[]{\includegraphics[width=0.25\textwidth]{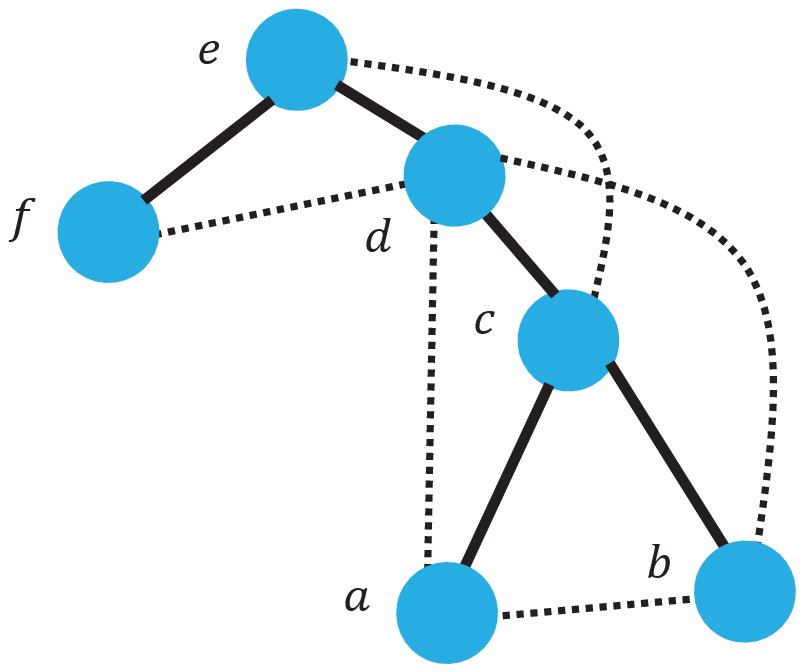}\label{fig:junction2}}
\vspace{-.25cm}
\caption{(a) Load Nodes in distribution tree ${\cal T}$ (b) Graphical Model $\cal GM$ for Nodal Voltages}
\label{fig:junction}
\vspace{-3mm}
\end{figure}

\textbf{Gaussian example}: To demonstrate the validity of the graphical model, we consider the linear DC power flow model (see Eq.~(\ref{PF_DC})) in tree ${\cal T}$ with operational edge set ${\cal E}_{\cal T}$. We consider the case where the vector of active load profiles ($P = p$) follows an uncorrelated multi-variate Gaussian distribution with diagonal covariance matrix $\Omega_{p}$. As linear functions of Gaussian random variables are Gaussian random variables, $\theta$ is a Gaussian random vector as well. Properties of multi-variate Gaussian distributions \cite{wainwright2008graphical} state that each non-zero off-diagonal term in the inverse covariance matrix $\Omega^{-1}_{\theta}$ of $\theta$ represents an edge in its representative graphical model. Using Eq.~(\ref{PF_DC}), $\Omega^{-1}_{\theta}$ is given by $\Omega^{-1}_{\theta}(a,b) = \sum_{c}H_{\beta}(a,c)\Omega^{-1}_{p}(c,c)H_{\beta}(b,c)$ where $H_{\beta}$ is the susceptance-weighted reduced Laplacian for ${\cal T}$. We now have
\begin{align} 
\Omega^{-1}_{\theta} (a,b)=\begin{cases}&H_{\beta}(a,b)H_{\beta}(a,a)\Omega^{-1}_{p}(a,a)\nonumber\\
&+ H_{\beta}(a,b)H_{\beta}(b,b)\Omega^{-1}_{p}(b,b) ~~\text{if $(ab) \in {\cal E}_{\cal T}$},\\
&H_{\beta}(a,c)H_{\beta}(b,c)\Omega^{-1}_{p}(c,c)~~\text{if $(ac), (bc) \in {\cal E}_{\cal T}$},\\
&0 ~~\text{otherwise} \end{cases} \label{treeinv}
\end{align}
Thus the graphical model for ${\theta}$ has the original operational edges ${\cal E}_{\cal T}$ in tree ${\cal T}$ and additional edges between two-hop neighbors, as suggested by Lemma~\ref{lemma1} and Corollary~\ref{corollary1}.

In the next section, we discuss the use of conditional independence tests to determine the presence of edges in ${\cal T}$. It is worth mentioning that all results in this section extend directly to distribution grids with multiple disjoint trees as the probability distribution for voltages in each tree is independent of others and hence the overall distribution can be factorized into product terms.

\section{Conditional Independence based Topology Learning}
\label{sec:learning}
Consider a distribution tree ${\cal T}$ with a substation as reference node and given set of all possible edges ${\cal E}$ (operational or open). The goal of topology learning scheme is to determine the operational edge set ${\cal E}_{\cal T}$ between load nodes ${\cal V}_{\cal T}$ using load voltage measurements in an invertible power flow model. Our learning approach uses results on conditional independence of voltage measurements in the graphical model $\cal GM$ as discussed below.

\begin{theorem}\label{condind}
Let distribution tree ${\cal T}$ be such that the p.d.f. of load voltages satisfies the graphical model in Lemma~\ref{lemma1}. The following conditional independence results hold:
\begin{itemize}
\item If $(ab)$ is an operational edge between non-leaf nodes $a$ and $b$ in $\cal T$, there exists distinct nodes $c$ and $d$ with ${\cal P}(V_c,V_d|V_a,V_b) = {\cal P}(V_c|V_a,V_b){\cal P}(V_d|V_a,V_b)$.
\item If $(ab)$ is \emph{not} an operational edge, then ${\cal P}(V_c,V_d|V_a,V_b) = {\cal P}(V_c|V_a,V_b){\cal P}(V_d|V_a,V_b)$ does \emph{not} hold for any nodes $c$ and $d$ distinct from $a$ and $b$ in ${\cal T}$.
\item If nodes $a$ and $b$ are separated by \emph{greater than two hops}, then there exists at least one operational edge $(cd)$ in $\cal T$ with ${\cal P}(V_a,V_b|V_c,V_d) = {\cal P}(V_a|V_c,V_d){\cal P}(V_b|V_c,V_d)$.
\end{itemize}
\end{theorem}
\begin{proof}
In tree $\cal T$, for any non-leaf nodes $a$ and $b$ connected by edge $(ab)$ there exists distinct nodes $c$ and $d$ that are on opposite side of $(ab)$ and separated by greater than two hops. In the corresponding graphical model ${\cal GM}$, nodes $b$ and $c$ belong to separate disjoint groups when nodes $a$ and $b$ are removed. Their voltages are thus conditionally independent given nodes $a$ and $b$. For the second statement, let $(ab)$ not be an edge in $\cal T$. Then there exists a path between any two nodes in $\cal GM$ even if both nodes $a$ and $b$ are removed as $\cal GM$ contains edges between the neighbors of either node in $\cal T$. Thus, removal of $a$ and $b$ doesn't lead to conditional independence of any other node pair. Finally, for the third statement, nodes $a$ and $b$ are separated by greater than two hops in $\cal T$. Thus, they exist on either side of edge $(cd)$ where $c$ and $d$ are the first two nodes on the path from $a$ to $b$. Using the first statement, ${\cal P}(V_a,V_b|V_c,V_d) = {\cal P}(V_a|V_c,V_d){\cal P}(V_b|V_c,V_d)$.
\end{proof}

In Fig.~\ref{fig:junction2}, voltages at nodes $f$ and $c$ will become conditionally independent given nodes $e$ and $d$. However, given nodes $e$ and $c$, no other pair becomes conditionally independent as the graphical model $\cal GM$ is not disconnected due to removal of $e$ and $c$. Note that the first conditional independence result in Theorem~\ref{condind} identifies operational edges between non-leaf nodes in ${\cal T}$. The final statement provides a technique to determine operational edges to each leaf node using discovered edges between its parent and other non-leaf nodes.

To distinguish all edges uniquely, \emph{we assume that $\cal T$ has depth greater than three, i.e., there exists a path of length at least four in $\cal T$}. This is necessary as no pair of non-leaf nodes exist for testing conditional independence results if depth is $2$. Further for depth $3$ (exactly two non-leaf nodes), it can be shown that the leaf nodes connected to each non-leaf node can be exchanged without changing the conditional independence tests. We omit the detailed proof for brevity. Under this depth assumption, we now present Algorithm $1$ for topology learning in distribution grid tree ${\cal T}$.

\begin{algorithm}
\caption{Topology Learning for Grid Tree ${\cal T}$}
\textbf{Input:} $m$ complex voltage observations $V^j ~(1\leq j \leq m)$ at all load nodes in set ${\cal V}_{\cal T}$, set of all candidate lines ${\cal E}$\\
\textbf{Output:} Operational edge set ${\cal E}_{\cal T} \subset {\cal E}$\\
\begin{algorithmic}[1]
\State ${\cal V}_p \gets \emptyset$ 
\ForAll{$(ab) \in {\cal E}$} \label{nonleaf1}
\ForAll{$(ac), (db) \in {\cal E}$, $c,d$ distinct from $a,b$ }
\If {${\cal P}(V_c,V_d|V_a,V_b) ={\cal P}(V_c|V_a,V_b){\cal P}(V_d|V_a,V_b)$}
\State ${\cal E}_{\cal T} \gets {\cal E}_{\cal T} \cup \{(ab)\}$, ${\cal V}_p \gets {\cal V}_p \cup \{a,b\}$
\EndIf
\EndFor
\EndFor \label{nonleaf2}
\State Sort nodes in ${\cal V}_p$ with increasing degree
\State ${\cal V}_{leaf} \gets {\cal V} -{\cal V}_p$ 
\For {\text{$i=1$ to $|{\cal V}_p|$}}
\State Pick node $a$ of degree $1$ in ${\cal V}_p$ \label{leaf1}
\State Pick nodes $b$ and $c$ such that $(ab),(bc) \in {\cal E}_{\cal T}$
\ForAll{$l \in {\cal V}_{leaf}, (al) \in {\cal E}$}
\If {${\cal P}(V_l,V_c|V_a,V_b) = {\cal P}(V_l|V_a,V_b){\cal P}(V_c|V_a,V_b)$}
\State ${\cal E}_{\cal T} \gets {\cal E}_{\cal T} \cup \{(al)\}$
\State ${\cal V}_{leaf} \gets {\cal V}_{leaf} - \{l\}$
\EndIf
\EndFor \label{leaf2}
\State Remove node $a$ from ${\cal V}_p$ and update degree of others
\EndFor
\end{algorithmic}
\end{algorithm}

\textbf{Working}: Note that we compute the conditional independence among nodes in the permissible neighborhood of each potential edge instead of all possible node combinations. Thus, we only test conditional independence of nodes $c$ and $d$ given $a$ and $b$ such that $(ac),(cd), (db)$ are in ${\cal E}$. This is sufficient as an operational edge between non-leaf nodes $a$ and $b$ will make at least two nodes in their respective neighborhoods conditional independent. Here, Steps~\ref{nonleaf1} to~\ref{nonleaf2} discover such edges between non-leaf nodes using the first statement in Theorem~\ref{condind}. The set ${\cal V}_p$ includes the non-leaf nodes that have been discovered. Next, Step~\ref{leaf1} to~\ref{leaf2} iteratively identifies connections to potential leaf nodes in set ${\cal V}_{leaf}$ using the final statement in Theorem~\ref{condind}. Note that in each iteration, we search for leaf nodes $l$ in the neighborhood of a non-leaf node $a$ with one discovered edge to other non-leaf nodes in ${\cal V}_p$. This is done to avoid cases where $l$ is not $a$'s child but still satisfies the third conditional test in Theorem~\ref{condind}. Once leaf nodes of node $a$ are discovered, it is removed from the set of potential parents ${\cal V}_p$ and another parent is selected.

\textbf{Complexity}: Each conditional independence test in Algorithm $1$ is restricted to $4$ nodes only. This is distinct from general graph learning where number of variables in conditional independence tests can scale with the graph size. Further, as each voltage is a complex valued quantity, we can reduce the complexity by computing the conditional independence of voltage magnitudes or phase angles at two nodes given the complex voltages at two others. This reduces the computation to $6$ real-valued random variables instead of $8$ earlier. If magnitude and phase angles are discrete quantities taking $p$ possible values each, the complexity of one conditional independence test is $O(p^6)$. As we deal with continuous random variables, the complexity function $C$ will depend on the estimation technique used as discussed in the next section. However $C$ will always be independent of the grid size. Determination of edges between non-leaf nodes requires $O(D_{max}^2)$ tests per edge in ${\cal E}$, where $D_{max}$ is the maximum degree of any node in ${\cal E}$. As the total number of non-leaf nodes is less than $|{\cal V}_{\cal T}|$, sorting them according to their degree has complexity $O(|{\cal V}_{\cal T}|\log|{\cal V}_{\cal T}|)$. Finally determining the edges connected to leaf nodes require total $O(|{\cal V}_{\cal T}|D_{max})$ tests in the neighborhood of non-leaf nodes. As $|{\cal V}_{\cal T}| \leq |{\cal E}|$, the overall complexity of Algorithm $1$ is terms of the distribution grid size is thus $O(|{\cal E}|D_{max}^2+ |{\cal V}_{\cal T}|\log|{\cal V}_{\cal T}|)$. Note that we analyze the complexity in terms of the size of set ${\cal E}$, as often the number of permissible edges is less than that in a complete graph. If $\cal E$ corresponds to a complete graph, ${\cal E} = O(|{\cal V}_{\cal T}|^2)$ and $D_{max} =O(|{\cal V}_{\cal T}|)$, so the algorithm scales as $O(|{\cal V}_{\cal T}|^4)$ in terms of the graph size.

In the general case where the distribution grid is composed of multiple disjoint trees, voltages at two nodes $c$ and $d$ of different trees will be independent of each other. Thus, while reconstructing the operational topology in the general case, we need to check if nodes $c$ and $d$ are \emph{unconditionally} independent. If the unconditional independence test holds true, they will be determined as belonging to separate trees and hence excluded from remaining conditional independence tests in Algorithm $1$ to determine operational edges. The general case will be addressed in detail in future work. In the next section, we discuss computation of the conditional independence test for continuous random variables.

\section{Computation of Conditional Independence Test}
\label{sec:conditional}
Algorithm $1$ checks if the voltage measurements at two nodes $c$ and $d$ are independent conditioned on knowledge of voltage measurements at two other nodes $a$ and $b$. As these measurements are continuous random variables with arbitrary distribution, testing their conditional independence is a non-trivial task. Non-parametric techniques to determine conditional independence in literature include ones based on distances between estimated conditional densities \cite{su2008nonparametric} or between characteristic functions \cite{su2003consistent}. Another technique \cite{margaritis2005distribution} involves binning the domain of conditioned random variables and then treating them like discrete variables. However, these techniques require large number of samples for correct estimation. A much cited line of work \cite{fukumizu2004dimensionality,gretton2007kernel,zhang2012kernel} has focussed on kernel-based techniques for non-parametric conditional independence test. Here conditional independence is characterized using covariance operators in Reproducing Kernel Hilbert Spaces (RKHS) and the Hilbert-Schmidt norm of the normalized conditional cross-covariance operator is used as a measure of conditional independence. When the random variables tested are conditionally independent, the corresponding Hilbert-Schmidt norm vanishes and identifies it. We follow the kernel-based conditional independence test in \cite{gretton2007kernel} in this paper. The matlab code for it can be found at \cite{zhang2012kernel}. Next we present simulation results of our topology learning algorithm on test radial distribution networks.

\section{Simulations Results}
\label{sec:simulations}
We demonstrate the performance of Algorithm $1$ in extracting the operational edge set ${\cal E}_{\cal T}$ of tree grid ${\cal T}$ from a loopy original edge set ${\cal E}$. We consider a modified tree distribution network \cite{testcase2,radialsource} with $19$ load nodes and one substation as shown in Fig.~\ref{fig:case}. For our simulations, we consider active load profiles to follow uncorrelated Gaussian random variables. Note that the algorithm is independent of the each load's distribution. We use DC PF model (see Eq.~(\ref{PF_DC})) to generate nodal voltage phase angle samples for the radial operational topology specified by ${\cal E}_{\cal T}$. As stated in the paper, the algorithm will work with other linear and invertible power flow models with separable Jacobian determinants as well. We then add $20$ additional edges at random locations to construct the loopy set of potential edges ${\cal E}$ that is used as input by the observer. In the test radial network depicted in Fig.~\ref{fig:case}, solid and dashed lines represent operational and open lines respectively. Fig.~\ref{fig:resultDC} plots the average errors (relative to the number of operational edges) generated by our algorithm for different samples sizes in the $19$ bus case. Note that increasing the number of samples leads to lesser number of errors for all tolerance values used in the conditional independence detection test (see Theorem~\ref{condind}). Further observe that a weaker tolerance $(4.5*10^{-7})$ produces the least number of errors for lower sample sizes. However for higher sample sizes (more accurate conditional independence tests), the majority of errors arise from selection of non-operational edges (false positives). Thus, the greatest accuracy is reached at a tighter tolerance $(3.5*10^{-7})$.
\begin{figure}[!bt]
\centering
\includegraphics[width=0.36\textwidth,height =.30\textwidth]{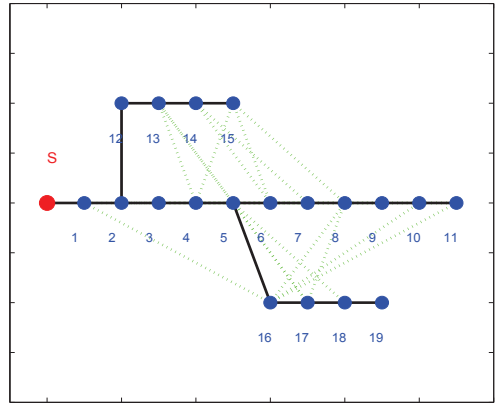}
\vspace{-.10cm}
\caption{Layouts of test distribution grid. The red circles represent substations (marked as $S$). The blue circles represent load nodes that are numbered. Black lines represent operational edges. The additional open lines are represented by dotted green lines.}
\label{fig:case}
\vspace{-2mm}
\end{figure}
\begin{figure}[!bt]
\centering
\includegraphics[width=0.42\textwidth,height = .38\textwidth]{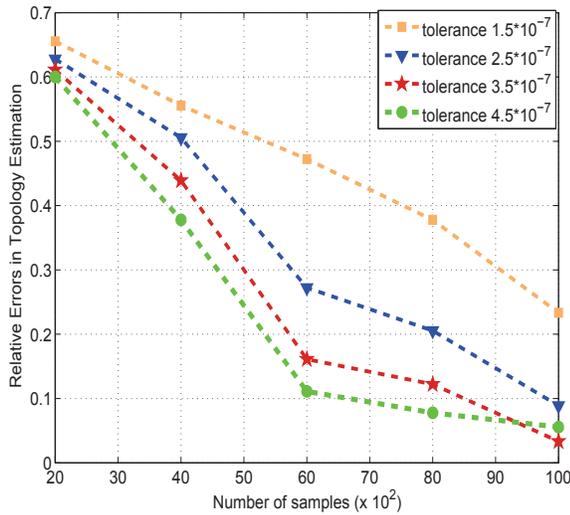}
\caption{Accuracy of Topology Learning Algorithm with number of measurement samples for test system in Fig.~\ref{fig:case}. The tolerance values are used to determine if a conditional independence test is positive.}
\label{fig:resultDC}
\vspace{-3mm}
\end{figure}

\section{Conclusion}
\label{sec:conclusion}
Efficient estimation of the operational topology of the distribution grid can benefit several applications. In this manuscript, we have presented an algorithm that uses non-parametric conditional independence tests to determine the operational edges in the radial grid. The algorithm is based on the structural features of the graphical model for voltage measurements. We show that the graphical model despite being loopy can be separated into independent sets conditioned on operational edges in the tree grid. We analyze the working of our topology learning algorithm and show that its computational complexity is a polynomial function of the grid size.

The primary benefit of our learning algorithm is that it uses both physical power flow laws and ideas from general graph learning. The learning algorithm is applicable for all lossless linear approximations as well as other invertible power flow models whose Jacobian determinants satisfy a separability criteria. Further, only nodal voltage measurements are needed as input to the algorithm and no flow or injection information is necessary. In fact the learning algorithm is independent of the line parameters (resistance and reactance) as well as explicit distributions and statistics associated with individual load profiles. We use kernel based tests for conditional independence within our algorithm and demonstrate its performance over a test distribution grid. Due to its general applicability to diverse operational conditions and models, our learning algorithm can thus be used for multiple applications like failure identification, security quantification and privacy aware flow optimization. Generalizing the algorithm for distribution grids with mutiple radial trees and developing a theoretical understanding of the sample complexity necessary for error free learning are potential future directions of research.

\bibliographystyle{IEEETran}
\bibliography{../../Bib/FIDVR,../../Bib/SmartGrid,../../Bib/voltage,../../Bib/trees}
\end{document}